\documentclass[a4paper,12pt]{article}
\usepackage{amsmath,amsthm,amssymb}
\title{Integral formula and upper estimate of\\ I and J-Bessel functions on Jordan algebras}
\author{Ryosuke Nakahama\thanks{Email: nakahama@ms.u-tokyo.ac.jp} \\
\textit{Graduate School of Mathematical Sciences, the University of Tokyo,} \\
\textit{3-8-1 Komaba Meguro-ku Tokyo 153-8914, Japan}}
\date{\today}
\makeatletter
\@addtoreset{equation}{section}
\makeatother

\newtheorem{theorem}{Theorem}[section]
\newtheorem*{theorem*}{Theorem}
\newtheorem{Definition}[theorem]{Definition}
\newtheorem{proposition}[theorem]{Proposition}
\newtheorem{corollary}[theorem]{Corollary}
\newtheorem*{corollary*}{Corollary}
\newtheorem{lemma}[theorem]{Lemma}

\newtheorem{remark}[theorem]{Remark}

\newcommand{\Z}{\mathbb{Z}}
\newcommand{\R}{\mathbb{R}}
\newcommand{\C}{\mathbb{C}}
\renewcommand{\H}{\mathbb{H}}
\newcommand{\Oct}{\mathbb{O}}
\renewcommand{\Re}{\operatorname{Re}}
\renewcommand{\Im}{\operatorname{Im}}
\newcommand{\tr}{\operatorname{tr}}
\newcommand{\Tr}{\operatorname{Tr}}
\newcommand{\Det}{\operatorname{Det}}
\newcommand{\rank}{\operatorname{rank}}
\begin{document}
\maketitle

\begin{abstract}
In this paper we give a new integral expression of I and J-Bessel functions on simple Euclidean Jordan algebras, 
integrating on a bounded symmetric domain. From this we easily get the upper estimate of Bessel functions. 
As an application we give an upper estimate of the integral kernel function of 
the holomorphic 1-dimensional semi-group acting on the space of 
square integrable functions on symmetric cones. \bigskip

\end{abstract}

\section{Introduction and main results}
In this paper we find in Theorem \ref{integral} a new integral expression of I and J-Bessel functions 
$\mathcal{I}_{\lambda}(x)$, $\mathcal{J}_{\lambda}(x)$ on a Jordan algebra $V$. 
J-Bessel functions are first introduced by Faraut and Travaglini \cite{FT} for special cases, 
associating to self-adjoint representations of Jordan algebras (see also (\ref{repBessel})), 
and generalized by Dib \cite{D} (for $V=\mathrm{Sym}(r,\R)$ case see also \cite{He} and \cite{Mu}). It is well-known that 
$\mathcal{I}_{\lambda}(x)$, $\mathcal{J}_{\lambda}(x)$ are the holomorphic functions on $V^{\C}$ for $\lambda$ 
in open dense subset of $\C$. On the other hand, for countable singular $\lambda$ they are still well-defined 
on certain subvarieties. These are defined by the series expansion (see Section \ref{main}), 
and satisfy the following differential equation
\[ \mathcal{B}_{\lambda}\mathcal{I}_{\lambda}-e\mathcal{I}_{\lambda}=0,\qquad
\mathcal{B}_{\lambda}\mathcal{J}_{\lambda}+e\mathcal{J}_{\lambda}=0 \]
where $\mathcal{B}_{\lambda}:C^2(V)\to C(V)\otimes V^{\C}$ is the $V^\C$-valued 2nd order differential operator 
defined in \cite[Section XV.2]{FK}, 
and $e$ is the unit element on $V$ (see \cite[Proposition 1.7]{D} or \cite[Theorem XV.2.6]{FK}). 
Also $\mathcal{I}_{\lambda}$ and $\mathcal{J}_{\lambda}$ have the following integral expression
\begin{gather}\label{tubeint}
\mathcal{I}_{\lambda}(x)=\frac{\Gamma_{\Omega}(\lambda)}{(2i\pi)^n}\int_{e+iV}e^{\tr w}e^{(w^{-1}|x)}\Delta(w)^{-\lambda}dw, \\
\mathcal{J}_{\lambda}(x)=\frac{\Gamma_{\Omega}(\lambda)}{(2i\pi)^n}\int_{e+iV}e^{\tr w}e^{-(w^{-1}|x)}\Delta(w)^{-\lambda}dw
\end{gather}
(see \cite[D\'efinition 1.2]{D} or \cite[Theorem XV.2.2]{FK}. For notations $\tr$, $(\cdot|\cdot)$, $\Delta$ and 
$\Gamma_{\Omega}(\lambda)$ see Section \ref{Jordan} and (\ref{Gamma})). 
There are some attempts to generalize these Bessel functions to operator-valued ones 
(see $e.g.$ \cite{DG} and references therein), but it is still not very well-understood. 
In this paper we only treat scalar-valued ones. 

Now we briefly state our theorem. Let $V$ be a simple Euclidean Jordan algebra 
($i.e.$, $V$ is one of the $\mathrm{Sym}(r,\R)$, $\mathrm{Herm}(r,\C)$, 
$\mathrm{Herm}(r,\H)$, $\R^{1,n-1}$ or $\mathrm{Herm}(3,\Oct)$). We assume $\dim V=n$, $\rank V=r$. We prove 

\begin{theorem}
For $\lambda\in\C$, $x\in\overline{\mathcal{X}_{\rank\lambda}}$ (see (\ref{subvar}) and (\ref{ranklambda})), 
take $k\in\Z_{\ge 0}$ such that $\Re\lambda+k>\frac{2n}{r}-1$. Then, we have the integral expressions
\begin{gather*}
\mathcal{I}_{\lambda}\left(x^2\right)
=c_{\lambda+k}\int_D{}_1F_1(-k,\lambda;-x,w)e^{2(x|\Re w)}h(w,w)^{\lambda+k-\frac{2n}{r}}dw, \\
\mathcal{J}_{\lambda}\left(x^2\right)
=c_{\lambda+k}\int_D{}_1F_1(-k,\lambda;-ix,w)e^{2i(x|\Re w)}h(w,w)^{\lambda+k-\frac{2n}{r}}dw, 
\end{gather*}
where $c_{\lambda}$ is a constant and ${}_1F_1(-k,\lambda;x,w)$ is a polynomial of degree $rk$ 
with respect to both $x$ and $w$. 
\end{theorem}

Here $\mathcal{X}_l$ are the $L=\mathrm{Str}(V^\C)_0$-orbits. $\overline{\mathcal{X}_l}$ are also characterized 
as the supports of some distributions on $V^\C$ (see \cite{BSZ} and (\ref{suppofdistr})). 
$D\subset V^{\C}$ is the \textit{bounded symmetric domain} and $h(w,w)$ is the \textit{generic norm} on $V^{\C}$ 
(see Section \ref{Jordan}). For the explicit forms of $c_\lambda$ and ${}_1F_1(-k,\lambda;x,w)$ see 
Theorem \ref{integral}. Especially if $\Re\lambda>\frac{2n}{r}-1$ we can take $k=0$ and 
\[ \mathcal{I}_{\lambda}\left(x^2\right)
=\frac{1}{\pi^n}\frac{\Gamma_{\Omega}(\lambda)}{\Gamma_{\Omega}\left(\lambda-\frac{n}{r}\right)}
\int_De^{2(x|\Re w)}h(w,w)^{\lambda-\frac{2n}{r}}dw \]
and $\mathcal{J}_{\lambda}$ is similar. 

Now $D$ is naturally identified with 
$G/K=\mathrm{Bihol}(D)/\mathrm{Stab}(0)=\mathrm{Co}(V)_0/\mathrm{Aut}_{\mathrm{JTS}}(V)_0$. 
For $\lambda>\frac{2n}{r}-1$, the universal covering group $\tilde{G}$ acts unitarily on 
$\mathcal{O}(D)\cap L^2(D,h(w,w)^{\lambda-\frac{2n}{r}}dw)$ by left translation. 
This defines the holomorphic discrete series representation of $\tilde{G}$. 
This is analytically continued with respect to $\lambda\in\C$, and become unitary when $\lambda\in\mathcal{W}$, 
the \textit{(Berezin--)Wallach set} (see (\ref{Wallach}) and \cite{W}, \cite{B}). 
The trivial representation corresponds to $\lambda=0$. 

From now we set $V=\R$. Let $I_{\lambda}(x)$ be the classical I-Bessel function (see \cite[(4.12.2)]{AAR}), 
and we set $\tilde{I}_{\lambda}(x)=\left(\frac{x}{2}\right)^{-\lambda}I_{\lambda}(x)$. 
Then $\tilde{I}_{\lambda}$ and $\mathcal{I}_{\lambda}$ on $\R$ are related as
\[ \tilde{I}_{\lambda}(x)=\frac{1}{\Gamma(\lambda+1)}\mathcal{I}_{\lambda+1}\left(\frac{x^2}{4}\right). \]
Therefore the above theorem is rewritten as 
\[ \tilde{I}_{\lambda}(x)=\frac{\lambda+k}{\pi\Gamma(\lambda+1)}
\int_{|w|<1}{}_1F_1(-k,\lambda+1;-xw)e^{x\Re w}\left(1-|w|^2\right)^{\lambda+k-1}dw. \]
where ${}_1F_1(-k,\lambda+1;x)$ is the classical hypergeometric polynomial. 
This formula seems to be new even for $V=\R$ case. On the other hand, 
the formula (\ref{tubeint}) is rewritten as 
\[ \tilde{I}_{\lambda}(x)=\frac{1}{2i\pi\lambda}\int_{1+i\R}e^{w+\frac{x^2}{w}}w^{-\lambda-1}dw. \]
These two integral formulas are mutually independent, and cannot easily deduce one from another. 

Again let $V$ be a general Jordan algebra. Since $D$ is bounded, we can prove from this formula the following corollary. 

\begin{corollary}\label{upperboundintro}
For $\lambda\in\C$, $x\in\overline{\mathcal{X}_{\rank\lambda}}$, if $\Re\lambda+k>\frac{2n}{r}-1$ 
for some $k\in\Z_{\ge 0}$, then there exists a positive constant $C_{\lambda,k}>0$ such that 
\[ |\mathcal{I}_{\lambda}(x^2)|\le C_{\lambda,k}\left(1+|x|_1^{rk}\right)e^{2|\Re x|_1}, \qquad
|\mathcal{J}_{\lambda}(x^2)|\le C_{\lambda,k}\left(1+|x|_1^{rk}\right)e^{2|\Im x|_1} \]
where $|x|_1$ is the norm defined in Definition \ref{pnormdef}. 
\end{corollary}

In \cite[Lemma 3.1]{M} an upper estimate of $\mathcal{J}_{\lambda}(x)$ is given by another method, 
but our estimate is sharper. For detail see Remark \ref{compare}. 
When $V=\R$, this corollary implies that if $\Re \lambda>-k$ for some $k\in\Z_{\ge 0}$, 
\[ |\tilde{I}_{\lambda}(x)|
=\frac{1}{|\Gamma(\lambda+1)|}\left|\mathcal{I}_{\lambda+1}\left(\frac{x^2}{4}\right)\right|
\le C_{\lambda,k}'\left(1+|x|^{k}\right)e^{|\Re x|}. \]
On the other hand, we have the asymptotic expansion 
\[ \tilde{I}_{\lambda}(x)\sim\frac{\left(\frac{x}{2}\right)^{-\lambda}}{\sqrt{2\pi x}}
\left(e^{x}\sum_{m=0}^\infty\frac{(-1)^m(\lambda,m)}{(2x)^m}
+e^{-x+\left(\lambda+\frac{1}{2}\right)\pi i}\sum_{m=0}^\infty\frac{(\lambda,m)}{(2x)^m}\right) \]
where $(\lambda,m)$ are some numbers (see \cite[(4.12.7)]{AAR}), and this implies that 
\[ |\tilde{I}_{\lambda}(x)|
\le C_{\lambda}''\left(1+|x|^{\max\left\{-\lambda-\frac{1}{2},0\right\}}\right)e^{|\Re x|}. \]
Therefore our result is not the sharpest when $\Re\lambda\le 0$, but it still seems to be 
sufficiently sharp. 

This paper is organized as follows: In Section 2, we recall some notations and facts about Euclidean Jordan algebras. 
In Section 3 we prove our main theorem, the integral formula and upper estimates. In Section 4, as an application 
of the inequality (Corollary \ref{upperboundintro}), we give an upper estimate of the integral kernel function 
of the 1-dimensional semigroup on the functions on the symmetric cones.

\section{Preliminaries}
\subsection{Simple Euclidean Jordan algebras}\label{Jordan}
Let $V$ be a simple Euclidean Jordan algebra of dimension $n$, rank $r$. We denote the unit element by $e$. 
Also let $V^{\C}$ be its complexification. For $x,y,z\in V^{\C}$, we write 
\begin{align*}
L(x)y&:=xy,\\ x\Box y&:=L(xy)+[L(x),L(y)],\\
P(x,z)&:=L(x)L(z)+L(z)L(x)-L(xz),\\ P(x)&:=P(x,x)=2L(x)^2-L(x^2),\\
B(x,y)&:=I_{V^{\C}}-2x\Box\bar{y}+P(x)P(\bar{y})
\end{align*}
where $y\mapsto\bar{y}$ is the complex conjugation with respect to the real form $V$. 
Also, we write 
\[ \{ x,y,z\}:=(x\Box \bar{y})z=P(x,z)\bar{y}=(x\bar{y})z+x(\bar{y}z)-(xz)\bar{y}. \]
Then $V^{\C}$ becomes a positive Hermitian Jordan triple system with this triple product. 

We denote the \textit{Jordan trace} and the \textit{Jordan determinant} of the complex Jordan algebra $V^{\C}$ 
by $\tr(x)$ and $\Delta(x)$ respectively. 
Also let $h(x,y)$ be the \textit{generic norm} of the Jordan triple system $V^{\C}$. 
These can be expressed by $L(x)$, $P(x)$, and $B(x,y)$ (see \cite[Proposition III.4.2]{FK}, 
\cite[Part V, Proposition VI.3.6]{FKKLR}): 
\begin{gather*}
\Tr L(x)=\frac{n}{r}\tr(x),\\ \Det P(x)=\Delta(x)^{\frac{2n}{r}},\\ \Det B(x,y)=h(x,y)^{\frac{2n}{r}} 
\end{gather*}
where $\Tr$ and $\Det$ stand for the usual trace and determinant of complex linear operators on $V^{\C}$. 
Using the Jordan trace we define the inner product on $V^{\C}$:
\[ (x|y):=\tr(x\bar{y}),\qquad x,y\in V^{\C}. \]
Then this is positive definite since $V$ is Euclidean. Also we define the \textit{symmetric cone} $\Omega$ and the 
\textit{bounded symmetric domain} $D$ by
\begin{gather*}
\Omega:=\{x^2: x\in V,\ \Delta(x)\ne 0\}, \\
D:=(\mathrm{connected\ component\ of}\ \{w\in V^\C: h(w,w)>0\}\ \mathrm{which\ contains\ 0}). 
\end{gather*}
Then $\Omega$ is self-dual, $i.e.$, 
\[ \Omega=\{x\in V: (x|y)>0\; \text{for any}\; y\in\Omega\}, \]
and $D$ is biholomorphically equivalent to $V+\sqrt{-1}\Omega\subset V^\C$. 
  
Let $K_L$ and $K$ be the identity components of \textit{automorphism groups} of the Jordan algebra $V$ and 
the Jordan triple system $V^{\C}$. Similarly let $L$ and $L^{\C}$ be the identity components of 
\textit{structure groups} of $V$ and $V^{\C}$. Also let $G$ be the identity component of 
\textit{conformal group} of $V$: 
\begin{gather*}
K_L:=\mathrm{Aut}_{\mathrm{J.Alg}}(V)_0=\{ k\in GL(V): k(xy)=kx\cdot ky, \; \forall x,y\in V\}_0,\\
K:=\mathrm{Aut}_{\mathrm{JTS}}(V^{\C})_0=\{ k\in GL(V^{\C}): k\{ x,y,z\}=\{ kx,ky,kz\}, \; \forall x,y,z\in V^{\C}\}_0,\\
L:=\mathrm{Str}(V)_0=\{ l\in GL(V): l\{ x,y,z\}=\{ lx,{}^tl^{-1}y,lz\}, \; \forall x,y,z\in V\}_0,\\
L^{\C}:=\mathrm{Str}(V^{\C})_0=\{ l\in GL(V^{\C}): l\{ x,y,z\}=\{ lx,(l^*)^{-1}y,lz\}, \; \forall x,y,z\in V^{\C}\}_0,\\
G:=\mathrm{Co}(V)_0=\mathrm{Bihol}(D)_0\simeq \mathrm{Bihol}(V+\sqrt{-1}\Omega)_0
\end{gather*}
where ${}^tl$ and $l^*$ stand for the transpose with respect to the bilinear form $\tr(xy)$ and 
the sesquilinear form $\tr(x\bar{y})=(x|y)$. 
Then $\Omega$ and $D$ are naturally identified with $L/K_L$ and $G/K$ respectively. 
For the classification of these groups see \cite[Table 1]{HKMO} or \cite[Table 1]{M}.

\subsection{Spectral decomposition and some norms on $V^{\C}$}
From now on we fix a \textit{Jordan frame} $\{c_1,\ldots,c_r\}\subset V$, $i.e.$, 
\begin{gather*}
c_jc_k=\delta_{jk}c_j,\qquad \sum_{j=1}^rc_j=e, \\ 
\text{and if}\; d_{j1},d_{j2}\in V \; \text{satisfy}\; c_j=d_{j1}+d_{j2},\ d_{jk}d_{jl}=\delta_{kl}d_{jk},\; 
\text{then}\; d_{j1}=0\; \text{or}\; d_{j2}=0. 
\end{gather*}
Then for any $x\in V^{\C}$ there exist the unique numbers $t_1\ge\cdots t_r\ge 0$ and the element $k\in K$ such that 
$x=k\sum_{j=1}^rt_jc_j$ (\cite[Proposition X.3.2]{FK}). Using this, we define the $p$-norm on $V^{\C}$. 

\begin{Definition}\label{pnormdef}
For $1\le p\le \infty$ and for $x=k\sum_{j=1}^rt_jc_j\in V^{\C}$, we define
\[ |x|_p:=\begin{cases} \displaystyle \left(\sum_{j=1}^r|t_j|^p\right)^{\frac{1}{p}}&(1\le p<\infty), \\
\displaystyle \max_{j\in\{1,\ldots,r\}}|t_j|&(p=\infty). \end{cases} \]
\end{Definition}

For example, we have $(x|x)=|x|_2^2$. Also if $x\in\Omega$ then all eigenvalues 
(in the sense of Jordan algebras. For $V=\mathrm{Sym}(r,\R)$ or $\mathrm{Herm}(r,\C)$ this coincides with the 
usual one)
are positive and $|x|_1=\tr x$ holds. In addition, we can define $D$ by $D=\{ w\in V^\C: |w|_\infty<1\}$. 
This norm satisfies the following properties. 

\begin{proposition}[{\cite[Theorem V.4, V.5]{S}} for $V=\mathrm{Herm}(r,\C)$ case]\label{pnorm}
Let $1\le p,q\le \infty$ and $\frac{1}{p}+\frac{1}{q}=1$. Then the following statements hold. 
\begin{enumerate}
\item For $x,y\in V^{\C}$, $|(x|y)|\le |x|_p|y|_q$. 
\item For $x\in V^{\C}$, $\displaystyle |x|_p=\max_{y\in V^{\C}\setminus\{0\}}\frac{|(x|y)|}{|y|_q}$. 
\item $x\mapsto |x|_p$ is a norm on $V^{\C}$. 
\end{enumerate}
\end{proposition}

To prove this, we quote the following lemma (see \cite[Part V, Proposition VI.2.1]{FKKLR}): 

\begin{lemma}\label{flat}
For $x,y\in V^{\C}$, if $x\Box\bar{y}=y\Box\bar{x}$, then there exists an element $k\in K$ 
such that both $x$ and $y$ belong to $\R\mathchar`-\operatorname{span}\{kc_1,\ldots,kc_r\}$. 
\end{lemma}

\begin{proof}[Proof of Proposition \ref{pnorm}]
(1) We note that $\displaystyle |(x|y)|\le\max_{k\in K}|(kx|y)|=\max_{k\in K}\Re(kx|y)$ 
since $e^{i\theta}I_{V^{\C}}\in K$ for any $\theta\in\R$. 
We take $k_0\in K$ such that $\Re(kx|y)$ $(k\in K)$ attains its maximum at $k=k_0\in K$. We put $k_0x=:x_0$. 
Then for any $D\in\mathfrak{k}=Lie(K)$, 
\[ \left.\frac{d}{dt}\right|_{t=0}\Re(e^{tD}x_0|y)=\Re(Dx_0|y)=0. \]
In the case when $D=u\Box \bar{v}-v\Box \bar{u}$ with $u,v\in V^{\C}$, 
\begin{align*}
0=&\Re((u\Box \bar{v})x_0|y)-\Re((v\Box \bar{u})x_0|y)=\Re((x_0\Box \bar{v})u|y)-\Re((x_0\Box \bar{u})v|y)\\
=&\Re(u|(v\Box \bar{x_0})y)-\Re(v|(u\Box \bar{x_0})y)=\Re(u|(y\Box \bar{x_0})v)-\Re(v|(y\Box \bar{x_0})u)\\
=&\Re((x_0\Box \bar{y})u|v)-\Re(v|(y\Box \bar{x_0})u)=\Re((x_0\Box \bar{y}-y\Box\bar{x_0})u|v). 
\end{align*}
Since $u,v\in V^{\C}$ are arbitrary and $(\cdot|\cdot)$ is non-degenerate, $x_0\Box \bar{y}=y\Box\bar{x_0}$. 
Therefore by Lemma \ref{flat} there exists $k\in K$ such that 
$x_0,y\in\R\mathchar`-\operatorname{span}\{kc_1,\ldots,kc_r\}$. 
Let $x=k'\sum_{j=1}^rt_jc_j$, $y=k\sum_{j=1}^rs_jc_j$. Then
\begin{align*}
|(x|y)|\le&\max_{k\in K}\Re(kx|y)=\Re(x_0|y)
=\Re\left(\left.k\sum_{j=1}^rt_jc_j\right|k\sum_{j=1}^rs_jc_j\right)\\
=&\sum_{j=1}^rt_js_j
\le\left(\sum_{j=1}^r|t_j|^p\right)^{\frac{1}{p}}\left(\sum_{j=1}^r|s_j|^q\right)^{\frac{1}{q}}=|x|_p|y|_q. 
\end{align*}
(2) ($\ge$) Clear from (1). \\
($\le$) For $x=k\sum_{j=1}^rt_jc_j\in V^{\C}$ $(t_1\ge\cdots t_r\ge 0)$, 
we find a $y\in V^{\C}$ which attains the equality. We set
\[ y:=\begin{cases} k\sum_{j=1}^rt_j^{p-1}c_j&(1\le p<\infty),\\ kc_1&(p=\infty).\end{cases} \]
Then,
\[ |y|_q=\begin{cases}\left(\sum_{j=1}^rt_j^{(p-1)q}\right)^{\frac{1}{q}}
=\left(\sum_{j=1}^rt_j^{p}\right)^{\frac{p-1}{p}}=|x|_p^{p-1}&(1<p<\infty),\\
1&(p=1,\infty),\end{cases} \]
and
\[ (x|y)=\begin{cases}\sum_{j=1}^rt_j^p=|x|_p^p=|x|_p|x|_p^{p-1}=|x|_p|y|_q&(1\le p<\infty),\\
t_1=|x|_{\infty}=|x|_{\infty}|y|_1&(p=\infty).\end{cases} \]
(3) Positivity and homogeneity are clear. For triangle inequality, by (2), for $x,y\in V^{\C}$, 
\[ |x+y|_p=\max_{|z|_q=1}|(x+y|z)|\le\max_{|z|_q=1}|(x|z)|+\max_{|z|_q=1}|(y|z)|=|x|_p+|y|_p \]
and this completes the proof. 
\end{proof}

We set 
\begin{equation}\label{subvar}
\mathcal{X}_l:=\left\{k\sum_{j=1}^lt_jc_j:k\in K,\ t_j>0\right\}=L^{\C}\cdot\sum_{j=1}^le_j\subset V^{\C}\quad 
(l=0,\ldots,r). 
\end{equation}
Then $\overline{\mathcal{X}_l}=\mathcal{X}_0\cup\mathcal{X}_1\cup\ldots\cup\mathcal{X}_l$ holds. 
$\overline{\mathcal{X}_l}$ are also characterized as the supports of the distributions which are the 
analytic continuation of $|\Delta(x)|^{2\left(\lambda-\frac{n}{r}\right)}dx$: 
\begin{equation}\label{suppofdistr}
\operatorname{supp} \left(\left.|\Delta(x)|^{2\left(\lambda-\frac{n}{r}\right)}dx\right|_{\lambda=l\frac{d}{2}}\right)
=\overline{\mathcal{X}_l}, \qquad l=0,1,\ldots,r-1
\end{equation}
(see \cite[Proposition 5.5]{BSZ}).

\subsection{Peirce decomposition and generalized power function}
As before we fix a Jordan frame $\{c_1,\ldots,c_r\}\subset V$. Then $V$ is decomposed as
\[ V=\bigoplus_{1\le j\le k\le r}V_{jk}\quad \text{where} \quad 
V_{jk}=\left\{ x\in V:L(c_l)x=\frac{\delta_{jl}+\delta_{kl}}{2}x\right\}. \]
Moreover $V_{jj}=\R c_j$ holds, and all $V_{jk}$'s ($j\ne k$) have the same dimension 
(see \cite[Theorem IV.2.1, Corollary IV.2.6]{FK}). 
We write $\dim V_{jk}=d$. 
Then $\dim V=n=r+\frac{1}{2}r(r-1)d$ holds. 

Let $V_{(l)}^{\C}:=\bigoplus_{1\le j\le k\le l}V_{jk}^{\C}$ $(l=1,\ldots,r)$ and $P_{(l)}$ be the orthogonal projection 
on $V_{(l)}^{\C}$. We denote by $\det_{(l)}(x)$ the Jordan determinant on the Jordan algebra $V_{(l)}^{\C}$. 
We set $\Delta_l(x):=\det_{(l)}(P_{(l)}(x))$ for $x\in V^{\C}$. For $\mathbf{s}=(s_1,\ldots,s_r)\in\C^r$, 
the \textit{generalized power function} on $V^{\C}$ is defined by
\[ \Delta_{\mathbf{s}}(x):=\Delta_1^{s_1-s_2}(x)\Delta_2^{s_2-s_3}(x)\cdots\Delta_{r-1}^{s_{r-1}-s_r}(x)
\Delta_r^{s_r}(x). \]
Then, the \textit{Gindikin Gamma function} and \textit{Pochhammer symbol} are defined as follows: 
for $\mathbf{s}\in \C^r$ and $\mathbf{m}\in (\Z_{\ge 0})^r$, 
\begin{equation}\label{Gamma}
\Gamma_{\Omega}(\mathbf{s}):=\int_{\Omega}e^{-\tr(x)}\Delta_{\mathbf{s}}(x)\Delta(x)^{-\frac{n}{r}}dx,\qquad
(\mathbf{s})_{\mathbf{m}}:=\frac{\Gamma_{\Omega}(\mathbf{s}+\mathbf{m})}{\Gamma_{\Omega}(\mathbf{s})}. 
\end{equation}
This integral converges for $\Re s_j>(j-1)\frac{d}{2}$, and both functions are extended meromorphically on $\C^r$ 
(see \cite[Theorem VII.1.1]{FK} or \cite[Theorem 2.1]{G}). Moreover, we have
\[ (\mathbf{s})_{\mathbf{m}}=\prod_{j=1}^r\left(s_j-(j-1)\frac{d}{2}\right)_{m_j}\quad
\text{where}\quad (s)_m=s(s+1)\cdots(s+m-1). \]
For $\mathbf{s}=(s_1,\ldots,s_r)\in\C^r$, we set $\mathbf{s}^*=(s_r,\ldots,s_1)$. Then we can prove easily 
\begin{equation}\label{Poch}
(\mathbf{s})_{\mathbf{m}+\mathbf{n}}=(\mathbf{s})_{\mathbf{m}}(\mathbf{s}+\mathbf{m})_{\mathbf{n}},\qquad
(-\mathbf{s}^*)_{\mathbf{m}}=(-1)^{|\mathbf{m}|}\left(\mathbf{s}-\mathbf{m}^*+\frac{n}{r}\right)_{\mathbf{m}^*}
\end{equation}
where $|\mathbf{m}|=m_1+\cdots+m_r$. Here we identify $\lambda\in\C$ and $(\lambda,\ldots,\lambda)\in\C^r$.

\subsection{Polynomials on $V^{\C}$}
We set $\Z_{++}^r:=\{\mathbf{m}=(m_1,\ldots,m_r)\in (\Z_{\ge 0})^r:m_1\ge m_2\ge\cdots m_r\ge 0\}$, and 
denote the space of holomorphic polynomials on $V^{\C}$ by $\mathcal{P}(V^{\C})$. 
For $\mathbf{m}\in\Z_{++}^r$, we define 
$\mathcal{P}_{\mathbf{m}}(V^{\C}):=\C\mathchar`-\operatorname{span}\{\Delta_{\mathbf{m}}\circ l:l\in L^{\C}\}$. 
Then clearly $\mathcal{P}_{\mathbf{m}}(V^{\C})$ becomes a $L^{\C}$-module. Moreover, we have 

\begin{theorem}[Hua--Kostant--Schmid, see {\cite[Theorem XI.2.4]{FK}}]\label{HKS}
\[ \mathcal{P}(V^{\C})=\bigoplus_{\mathbf{m}\in\Z_{++}^r}\mathcal{P}_{\mathbf{m}}(V^{\C}). \]
These $\mathcal{P}_{\mathbf{m}}(V^{\C})$'s are mutually inequivalent, and irreducible as $L^{\C}$-modules. 
\end{theorem}

Since $\Delta_l$ vanishes on $\overline{\mathcal{X}_{l-1}}$, all polynomials in $\mathcal{P}_{\mathbf{m}}(V^{\C})$ 
vanish on $\overline{\mathcal{X}_{l-1}}$ if and only if $m_l\ne 0$. 

We write $d_{\mathbf{m}}:=\dim\mathcal{P}_{\mathbf{m}}(V^{\C})$, and 
$\Phi_{\mathbf{m}}(x):=\int_{K_L}\Delta_{\mathbf{m}}(kx)dk$. 
Then the $K_L$-fixed subspace in $\mathcal{P}_{\mathbf{m}}(V^{\C})$ is spanned by $\Phi_{\mathbf{m}}$ 
(see \cite[Proposition XI.3.1]{FK}).

\subsection{Inner products on $\mathcal{P}(V^{\C})$}
For $f,g\in \mathcal{P}(V^{\C})$, we denote the \textit{Fischer inner product} by $\langle f,g\rangle_F$: 
\[ \langle f,g\rangle_F:=\frac{1}{\pi^n}\int_{V^{\C}}f(w)\overline{g(w)}e^{-(w|w)}dw
=\left.f\left(\frac{\partial}{\partial w}\right)\bar{g}(w)\right|_{w=0} \]
(For the second equality see \cite[Proposition XI.1.1]{FK}). 
Then the reproducing kernel of $\overline{\mathcal{P}(V^{\C})}^F$ (Hilbert completion of $\mathcal{P}(V^{\C})$) 
is given by $e^{(z|w)}$. We denote by $K^{\mathbf{m}}(z,w)=K^{\mathbf{m}}_w(z)$ the reproducing kernel of 
$\mathcal{P}_{\mathbf{m}}(V^{\C})$ with respect to $\langle\cdot,\cdot\rangle_F$. Then clearly, 
\[ e^{(z|w)}=\sum_{\mathbf{m}\in\Z_{++}^r}K^{\mathbf{m}}(z,w), \]
Also, by \cite[Proposition XI.3.3, Propsition XI.4.1.(ii)]{FK}, we have 
\begin{gather*}
K^{\mathbf{m}}(gz,w)=K^{\mathbf{m}}(z,g^*w) \qquad \text{for any}\; g\in \mathrm{Str}(V^{\C}), \\
K^{\mathbf{m}}_e(z)=\frac{1}{\Vert \Phi_{\mathbf{m}}\Vert_F^2}\Phi_{\mathbf{m}}(z)
=\frac{d_\mathbf{m}}{\left(\frac{n}{r}\right)_{\mathbf{m}}}\Phi_{\mathbf{m}}(z)
\end{gather*}
and 
\[ K^{\mathbf{m}}(x,\bar{x})=K^{\mathbf{m}}(x^2,e) \]
for $x\in V$, and therefore for any $x\in V^{\C}$ by analytic continuation. 

Also, for $\lambda>\frac{2n}{r}-1$, we denote the \textit{weighted Bergman inner product} on $D$ by
$\langle \cdot,\cdot \rangle_{\lambda}$: 
\[ \langle f,g\rangle_{\lambda}
:=\frac{1}{\pi^n}\frac{\Gamma_{\Omega}(\lambda)}{\Gamma_{\Omega}\left(\lambda-\frac{n}{r}\right)}
\int_Df(w)\overline{g(w)}h(w,w)^{\lambda-\frac{2n}{r}}dw. \]
Then, these two inner products are related as follows:

\begin{theorem}[Faraut--Kor\'anyi, see {\cite[Theorem XIII.2.7]{FK}}]
If $f,g\in\mathcal{P}(V^{\C})$ are decomposed as $f=\sum_{\mathbf{m}\in\Z_{++}^r}f_{\mathbf{m}}$, 
$g=\sum_{\mathbf{m}\in\Z_{++}^r}g_{\mathbf{m}}$ ($f_{\mathbf{m}},g_{\mathbf{m}}\in\mathcal{P}_\mathbf{m}(V^{\C})$), 
then
\begin{equation}\label{inner1}
\langle f,g\rangle_{\lambda}=\sum_{\mathbf{m}\in\Z_{++}^r}\frac{1}{(\lambda)_\mathbf{m}}
\langle f_\mathbf{m},g_\mathbf{m}\rangle_F. 
\end{equation}
\end{theorem}

Although the left hand side is only defined for $\lambda>\frac{2n}{r}-1$, 
the right hand side extends meromorphically for $\lambda\in\C$. 
Therefore we can redefine $\langle\cdot,\cdot\rangle_\lambda$ with this formula for any $\lambda\in\C$ by 
restricting the domain. For $\lambda\in\C$ we set 
\begin{align}\label{ranklambda}
\rank\lambda:=&\max\left\{ l\in\{0,1,\ldots,r\}:
(\lambda)_{\mathbf{m}}\ne 0\; \text{for any}\; \mathbf{m}\in\Z_{++}^r\cap\{ m_{l+1}=0\}\right\}\nonumber \\ 
=&\begin{cases} l&\text{if}\; \lambda\in\left(l\frac{d}{2}+\Z_{\le 0}\right)\setminus
\bigcup_{j=0}^{l-1} \left(j\frac{d}{2}+\Z_{\le 0}\right) \quad (l=0,1,\ldots,r-1),\\
r&\text{if}\; \lambda\notin\bigcup_{j=0}^{r-1} \left(j\frac{d}{2}+\Z_{\le 0}\right). \end{cases} 
\end{align}
For example, if $d=2$, $i.e.$, $V=\mathrm{Herm}(r,\C)$, then 
\[ \rank\lambda=\begin{cases}0&(\lambda\in\Z_{\le 0}),\\ l&(\lambda=l,\ l=1,\ldots,r-1),\\ 
r&(\lambda\notin r-1+\Z_{\le 0}).\end{cases} \] 
Then $\langle\cdot,\cdot\rangle_{\lambda}$ defines a sesquilinear form on 
$\bigoplus_{\mathbf{m}\in\Z_{++}^r,\ m_{\rank\lambda+1}=0}\mathcal{P}_{\mathbf{m}}(V^{\C})$. 
This form $\langle\cdot,\cdot\rangle_{\lambda}$ is positive definite if and only if 
\begin{equation}\label{Wallach}
\lambda\in\mathcal{W}:=
\left\{0,\frac{d}{2},\ldots,(r-1)\frac{d}{2}\right\}\cup\left((r-1)\frac{d}{2},\infty\right). 
\end{equation}
This set $\mathcal{W}$ is called the \textit{(Berezin--)Wallach set} (see \cite{W} or \cite{B}).

\subsection{Invariant differential operators}
For $\lambda\in\C$ and $k\in\Z_{\ge 0}$, we recall the differential operators $D^{(k)}$ from \cite[Section XIV.2]{FK}: 
\[ D^{(k)}(\lambda):=\Delta(x)^{\frac{n}{r}-\lambda}\Delta\left(\frac{\partial}{\partial x}\right)^k
\Delta(x)^{\lambda-\frac{n}{r}+k} \]
where $\Delta\left(\frac{\partial}{\partial x}\right)$ is the differential operator characterized by
$\Delta\left(\frac{\partial}{\partial x}\right)e^{(x|y)}=\Delta(y)e^{(x|y)}$. 
Then these operators commute with the $L^{\C}$-action 
($i.e.$, $D^{(k)}(\lambda)(f\circ l)=(D^{(k)}(\lambda)f)\circ l$ for $f\in\mathcal{P}(V^{\C})$ and $l\in L^{\C}$). 
Moreover, we have 

\begin{proposition}\label{invdiffop}
\[ D^{(k)}(\lambda)e^{(x|y)}=\sum_{\mathbf{m}\in\Z_{++}^r,\ |\mathbf{m}|\le rk}
(-1)^{|\mathbf{m}|}(-k)_{\mathbf{m}}(\lambda+\mathbf{m})_{k-\mathbf{m}}K^{\mathbf{m}}(x,y)e^{(x|y)}, \]
and if $(\lambda)_{\mathbf{m}}\ne 0$ for any $\mathbf{m}\in\Z_{++}^r$, $|\mathbf{m}|\le rk$, 
\[ D^{(k)}(\lambda)e^{(x|y)}=(\lambda)_k{}_1F_1(-k,\lambda;-x,y)e^{(x|y)} \]
where 
\begin{equation}\label{1F1}
{}_1F_1(-k,\lambda;-x,y):=\sum_{\mathbf{m}\in\Z_{++}^r,\ |\mathbf{m}|\le rk}
\frac{(-1)^{|\mathbf{m}|}(-k)_{\mathbf{m}}}{(\lambda)_{\mathbf{m}}}K^{\mathbf{m}}(x,y). 
\end{equation}
\end{proposition}

\begin{proof}
We follow the proof of \cite[Proposition XIV.1.5]{FK}. For $x\in\Omega$ and $\lambda<-k+1$, 
\begin{align*}
&D^{(k)}(\lambda)e^{(x|e)}=\Delta(x)^{\frac{n}{r}-\lambda}\Delta\left(\frac{\partial}{\partial x}\right)^k
\Delta(x)^{\lambda-\frac{n}{r}+k}e^{(x|e)}\\
=&\Delta(x)^{\frac{n}{r}-\lambda}\Delta\left(\frac{\partial}{\partial x}\right)^k
\frac{1}{\Gamma_{\Omega}\left(-\lambda+\frac{n}{r}-k\right)}\int_{\Omega}e^{(x|e-y)}
\Delta(y)^{-\lambda+\frac{n}{r}-k}\Delta(y)^{-\frac{n}{r}}dy\\
=&\Delta(x)^{\frac{n}{r}-\lambda}\frac{1}{\Gamma_{\Omega}\left(-\lambda+\frac{n}{r}-k\right)}
\int_{\Omega}e^{(x|e-y)}\Delta(e-y)^k\Delta(y)^{-\lambda-k}dy\\
=&\Delta(x)^{\frac{n}{r}-\lambda}\frac{1}{\Gamma_{\Omega}\left(-\lambda+\frac{n}{r}-k\right)}
\sum_{\mathbf{m}\in\Z_{++}^r,\ |\mathbf{m}|\le rk}d_{\mathbf{m}}
\frac{(-k)_{\mathbf{m}}}{\left(\frac{n}{r}\right)_\mathbf{m}}
\int_{\Omega}e^{(x|e-y)}\Phi_{\mathbf{m}}(y)\Delta(y)^{-\lambda-k}dy\\
=&\Delta(x)^{\frac{n}{r}-\lambda}\sum_{\mathbf{m}\in\Z_{++}^r,\ |\mathbf{m}|\le rk}d_{\mathbf{m}}
\frac{(-k)_{\mathbf{m}}}{\left(\frac{n}{r}\right)_\mathbf{m}}
\frac{\Gamma_{\Omega}\left(\mathbf{m}-\lambda+\frac{n}{r}-k\right)}
{\Gamma_{\Omega}\left(-\lambda+\frac{n}{r}-k\right)}\Phi_{\mathbf{m}}(x^{-1})\Delta(x)^{\lambda-\frac{n}{r}+k}e^{(x|e)}\\
=&\sum_{\mathbf{m}\in\Z_{++}^r,\ |\mathbf{m}|\le rk}
\frac{d_{\mathbf{m}}(-k)_{\mathbf{m}}\left(-\lambda+\frac{n}{r}-k\right)_{\mathbf{m}}}
{\left(\frac{n}{r}\right)_\mathbf{m}}\Phi_{k-\mathbf{m}^*}(x)e^{(x|e)}\\
=&\sum_{\mathbf{m}\in\Z_{++}^r,\ |\mathbf{m}|\le rk}
\frac{d_{k-\mathbf{m}^*}(-k)_{k-\mathbf{m}^*}\left(-\lambda+\frac{n}{r}-k\right)_{k-\mathbf{m}^*}}
{\left(\frac{n}{r}\right)_{k-\mathbf{m}^*}}\Phi_{\mathbf{m}}(x)e^{(x|e)}. 
\end{align*}
Here we used \cite[Lemma XI.2.3]{FK} at the 2nd and 5th equalities, and 
\cite[Corollary XII.1.3]{FK} at the 4th equality. 
At the 6th equality we used $\Phi_{\mathbf{m}}(x^{-1})\Delta(x)^k=\Phi_{k-\mathbf{m}^*}(x)$, 
which follows from the linear isomorphism 
$\mathcal{P}_{\mathbf{m}}(V^{\C})\to\mathcal{P}_{k-\mathbf{m}^*}(V^{\C})$, $p\mapsto \Delta(x)^kp(x^{-1})$. 
Now, $d_{\mathbf{m}}=d_{k-\mathbf{m}^*}$ holds by this isomorphism, and by (\ref{Poch}), 
\begin{gather*}
\frac{(-k)_{k-\mathbf{m}^*}}{\left(\frac{n}{r}\right)_{k-\mathbf{m}^*}}
=\frac{(-1)^{|k-\mathbf{m}^*|}\left(\frac{n}{r}+\mathbf{m}\right)_{k-\mathbf{m}}}
{\left(\frac{n}{r}\right)_{k-\mathbf{m}^*}}
=\frac{(-1)^{|k-\mathbf{m}^*|}\left(\frac{n}{r}\right)_k}
{\left(\frac{n}{r}\right)_{\mathbf{m}}\left(\frac{n}{r}\right)_{k-\mathbf{m}^*}}
=\frac{(-1)^{|k-\mathbf{m}^*|}(-k)_{\mathbf{m}}}{(-1)^{|\mathbf{m}|}\left(\frac{n}{r}\right)_{\mathbf{m}}},\\
\left(-\lambda+\frac{n}{r}-k\right)_{k-\mathbf{m}^*}=(-1)^{|k-\mathbf{m}^*|}(\lambda+\mathbf{m})_{k-\mathbf{m}}. 
\end{gather*}
Therefore, 
\[ D^{(k)}(\lambda)e^{(x|e)}=\sum_{\mathbf{m}\in\Z_{++}^r,\ |\mathbf{m}|\le rk}
(-1)^{|\mathbf{m}|}(-k)_{\mathbf{m}}(\lambda+\mathbf{m})_{k-\mathbf{m}}
\frac{d_{\mathbf{m}}}{\left(\frac{n}{r}\right)_{\mathbf{m}}}\Phi_{\mathbf{m}}(x)e^{(x|e)}. \]
By the $L^{\C}$-invariance of $D^{(k)}(\lambda)$, for $y\in\Omega$, 
\begin{align*}
&D^{(k)}(\lambda)e^{(x|y)}=D^{(k)}(\lambda)e^{(P(y^{\frac{1}{2}})x|e)}\\
=&\sum_{\mathbf{m}\in\Z_{++}^r,\ |\mathbf{m}|\le rk}
(-1)^{|\mathbf{m}|}(-k)_{\mathbf{m}}(\lambda+\mathbf{m})_{k-\mathbf{m}}
\frac{d_{\mathbf{m}}}{\left(\frac{n}{r}\right)_{\mathbf{m}}}\Phi_{\mathbf{m}}(P(y^{\frac{1}{2}})x)
e^{(P(y^{\frac{1}{2}})x|e)}\\
=&\sum_{\mathbf{m}\in\Z_{++}^r,\ |\mathbf{m}|\le rk}
(-1)^{|\mathbf{m}|}(-k)_{\mathbf{m}}(\lambda+\mathbf{m})_{k-\mathbf{m}}K^{\mathbf{m}}(x,y)e^{(x|y)}. 
\end{align*}
This holds for any $x,y\in V^{\C}$ and $\lambda\in\C$ by analytic continuation. 
The second equality follows from 
\[ (\lambda+\mathbf{m})_{k-\mathbf{m}}=\frac{(\lambda)_k}{(\lambda)_{\mathbf{m}}}. \qedhere \]
\end{proof}

Using these differential operators, we can calculate $\langle f,g\rangle_{\lambda}$ for $\lambda\in\C$: 
for $\Re\lambda+k>\frac{2n}{r}-1$ and 
$f,g\in\bigoplus_{\mathbf{m}\in\Z_{++}^r,\ m_{\rank\lambda+1}=0}\mathcal{P}_{\mathbf{m}}(V^{\C})$, 
\begin{equation}\label{inner2}
\langle f,g\rangle_{\lambda}=\begin{cases}
\displaystyle \frac{c_{\lambda+k}}{(\lambda)_k}\int_D(D^{(k)}(\lambda)f)(w)\overline{g(w)}
h(w,w)^{\lambda+k-\frac{2n}{r}}dw &(\rank\lambda=r)\\
\displaystyle \lim_{\mu\to\lambda}\frac{c_{\mu+k}}{(\mu)_k}\int_D(D^{(k)}(\mu)f)(w)\overline{g(w)}
h(w,w)^{\mu+k-\frac{2n}{r}}dw &(\rank\lambda<r) \end{cases}
\end{equation}
where $c_{\lambda}=\frac{1}{\pi^n}\frac{\Gamma_{\Omega}(\lambda)}{\Gamma_{\Omega}\left(\lambda-\frac{n}{r}\right)}$ 
(see \cite[Proposition XIV.2.2, Proposition XIV.2.5]{FK}). 
We can prove easily that this equality holds not only for polynomials, 
but also for holomorphic functions $f,g\in\mathcal{O}(D)$ with $D^{(k)}(\lambda)f$ and $g$ bounded on $\overline{D}$.

\section{Proof for main theorem}\label{main}
For $\lambda\in\C$ with $\rank\lambda=r$, the I and J-Bessel functions are defined by 
\begin{align*}
\mathcal{I}_{\lambda}(x):=&\sum_{\mathbf{m}\in \Z_{++}^r}
\frac{d_{\mathbf{m}}}{\left(\frac{n}{r}\right)_{\mathbf{m}}}\frac{1}{(\lambda)_{\mathbf{m}}}
\Phi_{\mathbf{m}}(x), \\
\mathcal{J}_{\lambda}(x):=&\sum_{\mathbf{m}\in \Z_{++}^r}
\frac{d_{\mathbf{m}}}{\left(\frac{n}{r}\right)_{\mathbf{m}}}\frac{(-1)^{|\mathbf{m}|}}{(\lambda)_{\mathbf{m}}}
\Phi_{\mathbf{m}}(x)=\mathcal{I}_{\lambda}(-x). 
\end{align*}
If $\rank\lambda<r$, then $(\lambda)_{\mathbf{m}}=0$ for some $\mathbf{m}$, so we cannot define these functions 
on entire $V^{\C}$. However, if $x\in\overline{\mathcal{X}_l}$, $\Phi_{\mathbf{m}}(x)=0$ for $m_{l+1}\ne 0$, 
and therefore for any $\lambda\in\C$ we can define I and J-Bessel functions for 
$x\in\overline{\mathcal{X}_{\rank\lambda}}$ (see (\ref{subvar}) and (\ref{ranklambda})) by 
\begin{align*}
\mathcal{I}_{\lambda}(x):=&\sum_{\mathbf{m}\in \Z_{++}^r,\ m_{\rank\lambda+1}=0}
\frac{d_{\mathbf{m}}}{\left(\frac{n}{r}\right)_{\mathbf{m}}}\frac{1}{(\lambda)_{\mathbf{m}}}
\Phi_{\mathbf{m}}(x), \\
\mathcal{J}_{\lambda}(x):=&\sum_{\mathbf{m}\in \Z_{++}^r,\ m_{\rank\lambda+1}=0}
\frac{d_{\mathbf{m}}}{\left(\frac{n}{r}\right)_{\mathbf{m}}}\frac{(-1)^{|\mathbf{m}|}}{(\lambda)_{\mathbf{m}}}
\Phi_{\mathbf{m}}(x)=\mathcal{I}_{\lambda}(-x). 
\end{align*}
Now we are ready to state the main theorem. 

\begin{theorem}\label{integral}
For $\lambda\in\C$, $x\in\overline{\mathcal{X}_{\rank\lambda}}$, take $k\in\Z_{\ge 0}$ such that 
$\Re\lambda+k>\frac{2n}{r}-1$. Then we have the integral expressions 
\begin{gather*}
\mathcal{I}_{\lambda}\left(x^2\right)
=c_{\lambda+k}\int_D{}_1F_1(-k,\lambda;-x,w)e^{2(x|\Re w)}h(w,w)^{\lambda+k-\frac{2n}{r}}dw, \\
\mathcal{J}_{\lambda}\left(x^2\right)
=c_{\lambda+k}\int_D{}_1F_1(-k,\lambda;-ix,w)e^{2i(x|\Re w)}h(w,w)^{\lambda+k-\frac{2n}{r}}dw. 
\end{gather*}
where
\[ c_{\lambda}=\frac{1}{\pi^n}\frac{\Gamma_{\Omega}(\lambda)}{\Gamma_{\Omega}\left(\lambda-\frac{n}{r}\right)},\quad
{}_1F_1(-k,\lambda;x,w)=\sum_{\shortstack{\scriptsize{$\mathbf{m}\in\Z_{++}^r$, $|\mathbf{m}|\le rk$,}\\ 
\scriptsize{$m_{\rank\lambda+1}=0$}}}
\frac{(-k)_{\mathbf{m}}}{(\lambda)_{\mathbf{m}}}K^{\mathbf{m}}(x,w). \]
\end{theorem}

When $\rank\lambda=r$, the definition of ${}_1F_1$ clearly coincides with the one in (\ref{1F1}). 

\begin{proof}
We calculate $\left\langle e^{(\cdot|\bar{x})}, e^{(\cdot|x)}\right\rangle_{\lambda}$ in two ways. By (\ref{inner1}), 
\begin{align*}
\left\langle e^{(\cdot|\bar{x})}, e^{(\cdot|x)}\right\rangle_{\lambda}
=&\left\langle\sum_{\mathbf{m}\in\Z^r_{++}}K^{\mathbf{m}}_{\bar{x}}, 
\sum_{\mathbf{n}\in\Z^r_{++}}K^{\mathbf{n}}_x\right\rangle_{\lambda}
=\sum_{\mathbf{m}\in\Z^r_{++}}\frac{1}{(\lambda)_\mathbf{m}}
\left\langle K^{\mathbf{m}}_{\bar{x}}, K^{\mathbf{m}}_x\right\rangle_F\\
=&\sum_{\mathbf{m}\in\Z^r_{++}}\frac{1}{(\lambda)_\mathbf{m}}K^{\mathbf{m}}(x,\bar{x})
=\sum_{\mathbf{m}\in\Z^r_{++}}\frac{1}{(\lambda)_\mathbf{m}}K^{\mathbf{m}}(x^2,e)\\
=&\sum_{\mathbf{m}\in\Z^r_{++}}\frac{1}{(\lambda)_\mathbf{m}}
\frac{d_\mathbf{m}}{\left(\frac{n}{r}\right)_\mathbf{m}}\Phi_{\mathbf{m}}(x^2)
=\mathcal{I}(x^2). 
\end{align*}
On the other hand, by (\ref{inner2}) and Proposition \ref{invdiffop}, 
\begin{align*}
\left\langle e^{(\cdot|\bar{x})}, e^{(\cdot|x)}\right\rangle_{\lambda}
=&\lim_{\mu\to\lambda}\frac{c_{\mu+k}}{(\mu)_k}\int_D\left(D^{(k)}(\mu)e^{(w|\bar{x})}\right)\overline{e^{(w|x)}}
h(w,w)^{\mu+k-\frac{2n}{r}}dw\\
=&\lim_{\mu\to\lambda}c_{\mu+k}\int_D{}_1F_1(-k,\mu;-x,w)e^{(w|\bar{x})}
\overline{e^{(w|x)}}h(w,w)^{\mu+k-\frac{2n}{r}}dw\\
=&c_{\lambda+k}\int_D{}_1F_1(-k,\lambda;-x,w)e^{2(x|\Re w)}h(w,w)^{\lambda+k-\frac{2n}{r}}dw. 
\end{align*}
The formula for $\mathcal{J}_{\lambda}\left(x^2\right)$ follows by replacing $x$ by $ix$. 
\end{proof}

From this theorem we can easily deduce the following corollary.  

\begin{corollary}\label{upperestimate}
For $\lambda\in\C$, $x\in\overline{\mathcal{X}_{\rank\lambda}}$, if $\Re\lambda+k>\frac{2n}{r}-1$ 
for some $k\in\Z_{\ge 0}$, then there exists a positive constant $C_{\lambda,k}>0$ such that 
\[ |\mathcal{I}_{\lambda}(x^2)|\le C_{\lambda,k}\left(1+|x|_1^{rk}\right)e^{2|\Re x|_1}, \qquad
|\mathcal{J}_{\lambda}(x^2)|\le C_{\lambda,k}\left(1+|x|_1^{rk}\right)e^{2|\Im x|_1} \]
where $|x|_1$ is the norm defined in Definition \ref{pnormdef}. 
\end{corollary}

\begin{proof}
By Proposition \ref{pnorm}, for $w\in D$, $x\in V^{\C}$, 
\[ |(\Re x|\Re w)|\le|\Re x|_1|\Re w|_{\infty}\le|\Re x|_1\frac{|w|_{\infty}+|\bar{w}|_{\infty}}{2}\le|\Re x|_1. \]
Also, since ${}_1F_1(-k,\lambda;-x,w)$ is a polynomial of degree $rk$ with respect to both $x$ and $w$, 
\[ |{}_1F_1(-k,\lambda;-x,w)|\le C_{\lambda,k}'\left(1+|x|_1^{rk}\right)\left(1+|w|_{\infty}^{rk}\right)
\le 2C_{\lambda,k}'\left(1+|x|_1^{rk}\right). \]
Therefore, by Theorem \ref{integral}, 
\begin{align*}
|\mathcal{I}_{\lambda}(x^2)|
&\le |c_{\lambda+k}|\int_D|{}_1F_1(-k,\lambda;-x,w)|e^{2(\Re x|\Re w)}h(w,w)^{\Re\lambda+k-\frac{2n}{r}}dw\\
&\le 2|c_{\lambda+k}|C_{\lambda,k}'\left(1+|x|_1^{rk}\right)e^{2|\Re x|_1}\int_Dh(w,w)^{\Re \lambda+k-\frac{2n}{r}}dw\\
&=C_{\lambda,k}\left(1+|x|_1^{rk}\right)e^{2|\Re x|_1}. 
\end{align*}
The proof for $\mathcal{J}_{\lambda}\left(x^2\right)$ is similar. 
\end{proof}

\begin{remark}\label{compare}
In \cite[Lemma 3.1]{M} M\"{o}llers gave another estimate of $\mathcal{J}_{\lambda}(x)$: 
\[ \left|\mathcal{J}_{\lambda}\left(x^2\right)\right|\le C\left(1+|x|_2^2\right)^{\frac{r(2n-1)}{4}}e^{2r|x|_2}
\qquad \text{for any}\; \lambda\in\mathcal{W},\ x\in\overline{\mathcal{X}_{\rank\lambda}}\subset V^\C. \]
However, our estimate is sharper because our leading term is given by $e^{2|\Im x|_1}$. 
Especially in our estimate $\mathcal{J}_{\lambda}(x)$ is uniformly bounded on $V$ if $\Re \lambda$ is sufficiently large. 
This difference comes from that of methods of proofs: in \cite{M} the Taylor expansion was used, 
while in this paper we use the integral formula. However, in general Taylor series is not strong enough 
for $L^\infty$ estimates. For example, the bound of cosine function is calculated as follows:
\[ |\cos x|=\left|\sum_{m=0}^{\infty}\frac{(-1)^m}{(2m)!}x^{2m}\right|
\le\sum_{m=0}^{\infty}\frac{1}{(2m)!}|x|^{2m}\le\sum_{m=0}^{\infty}\frac{1}{m!}|x|^m=e^{|x|}. \]
However, it is well-known that cosine function is bounded unformly on $\R$. So this bound is not sharp. 
\end{remark}

\section{Applications}
For $\lambda>\frac{n}{r}-1$, $t\in\C\setminus\pi i\Z$, $\Re t\ge 0$, we define a integral operator on $\Omega$: 
for a measurable function $\varphi:\Omega\to\C$, we define 
\[ \tau_{\lambda}(t)\varphi(x):=\frac{1}{\Gamma_{\Omega}(\lambda)}
\int_{\Omega}\varphi(y)\frac{e^{-\coth t(\tr x+\tr y)}}{\sinh^{r\lambda}t}\mathcal{I}_{\lambda}
\left(\frac{1}{\sinh^2t}P(x^{\frac{1}{2}})y\right)\Delta(y)^{\lambda-\frac{n}{r}}dy. \]
Since $\mathcal{I}_{\lambda}$ is $K$-invariant, by \cite[Lemma XIV.1.2]{FK} we can replace 
$P(x^{\frac{1}{2}})y$ by $P(y^{\frac{1}{2}})x$. 

\begin{remark}
For $\lambda>\frac{2n}{r}-1$, the Laplace transform 
\[ \mathcal{L}_{\lambda}:L^2(\Omega,\Delta(x)^{\lambda-\frac{n}{r}}dx)\longrightarrow
L^2(V+\sqrt{-1}\Omega,\Delta(\Im z)^{\lambda-\frac{2n}{r}}dz)\cap\mathcal{O}(V+\sqrt{-1}\Omega) \]
is defined by
\[ \mathcal{L}_{\lambda}\varphi(z):=\frac{2^n}{\Gamma_{\Omega}(\lambda)}\int_{\Omega}
e^{i(z|x)}\varphi(x)\Delta(2x)^{\lambda-\frac{n}{r}}dx. \]
Then we can prove by the similar method to \cite[Theorem XV.4.1]{FK} that
\begin{align*}
\mathcal{L}_{\lambda}\tau_{\lambda}(t)\mathcal{L}_{\lambda}^{-1}F(z)
=&\Delta(-\sin(it)z+\cos(it)e)^{-\lambda} \\
&\qquad \times F\left((\cos(it)z+\sin(it)e)(-\sin(it)z+\cos(it)e)^{-1}\right). 
\end{align*}
If $t$ is purely imaginary, then this coincides with the restriction of 
the holomorphic discrete series representation of 
the simple Hermitian Lie group $\mathrm{Bihol}(V+\sqrt{-1}\Omega)$, 
to the center of the maximal compact subgroup $\mathrm{Stab}(ie)$. 
That is, $\tau_{\lambda}$ can be regarded as the natural complexification of the action of 
$Z(\mathrm{Stab}(ie))\subset\mathrm{Bihol}(V+\sqrt{-1}\Omega)$. 
Especially, $\tau_{\lambda}(s)\tau_{\lambda}(t)=\tau_{\lambda}(s+t)$ holds for $\lambda>\frac{2n}{r}-1$. 
\end{remark}
\begin{remark}\label{Euclidrem}
Let $E$ be an Euclidean vector space of dimension $N$ with inner product $(\cdot|\cdot)_E$. 
Then the Hermite semigroup on $L^2(E)$ is given by 
\begin{equation}\label{Hermite}
\tilde{\tau}(t)f(\xi):=\frac{1}{(2\pi\sinh t)^{\frac{N}{2}}}
\int_{E}f(\eta)\exp\left(-\frac{1}{2}\coth t(|\xi|_E^2+|\eta|_E^2)+\frac{1}{\sinh t}(\xi|\eta)_E\right)d\eta
\end{equation}
for $f\in L^2(E)$, $t\in\C\setminus\pi i\Z$, $\Re t\ge 0$ (see, $e.g.$, \cite[Section 5.2]{F}). 
From now on we assume there exists an self-adjoint representation $\phi:V\to\mathrm{End}(E)$. 
We also assume $N>r(r-1)d$. Let $Q:E\to V$ be the quadratic map defined by
\[ (\phi(x)\xi|\xi)_E=(x|Q(\xi))_V \qquad \text{for any}\; x\in V,\ \xi\in E. \]
Let $\Sigma:=Q^{-1}(e)\subset E$ be the Stiefel manifold. Then we have 
\begin{equation}\label{repBessel}
\int_{\Sigma}e^{-i(\xi|\sigma)}d\sigma=\mathcal{J}_{\frac{N}{2r}}\left(Q\left(\frac{\xi}{2}\right)\right) 
\end{equation}
(see \cite[Proposition XVI.2.3]{FK}). 
We extend $Q$ to $Q:E^{\C}\to V^{\C}$ bilinearly. Then since $\mathcal{J}_{\lambda}(x)=\mathcal{I}_{\lambda}(-x)$ 
we have 
\[ \int_{\Sigma}e^{(\xi|\sigma)}d\sigma=\mathcal{I}_{\frac{N}{2r}}\left(Q\left(\frac{\xi}{2}\right)\right). \]
If $f\in L^2(E)$ is written as $f(\xi)=F\left(\frac{1}{2}Q(\xi)\right)$ with a function $F$ on $V$, 
then (\ref{Hermite}) can be rewritten as 
\begin{align*}
&\tilde{\tau}(t)f(\xi)=\frac{1}{(2\pi\sinh t)^{\frac{N}{2}}}
\int_{E}F\left(\frac{1}{2}Q(\eta)\right)\exp\left(-\frac{1}{2}\coth t(|\xi|_E^2+|\eta|_E^2)
+\frac{1}{\sinh t}(\xi|\eta)_E\right)d\eta\\
=&\frac{1}{(\pi\sinh t)^{\frac{N}{2}}}
\int_{E}F(Q(\eta))\exp\left(-\coth t\left(\frac{1}{2}|\xi|_E^2+|\eta|_E^2\right)
+\frac{\sqrt{2}}{\sinh t}(\xi|\eta)_E\right)d\eta\\
=&\frac{1}{\Gamma_{\Omega}(\frac{N}{2r})\sinh^{\frac{N}{2}}t}\int_{\Omega}\int_{\Sigma}
F(Q(\phi(y^{\frac{1}{2}})\sigma))
\exp\left(-\coth t\left(\frac{1}{2}|\xi|_E^2+|\phi(y^{\frac{1}{2}})\sigma|_E^2\right)\right)\\
&\qquad \times\exp\left(\frac{\sqrt{2}}{\sinh t}(\xi|\phi(y^{\frac{1}{2}})\sigma)_E\right)
\Delta(y)^{\frac{N}{2r}-\frac{n}{r}}d\sigma dy\\
=&\frac{1}{\Gamma_{\Omega}(\frac{N}{2r})}\int_{\Omega}\int_{\Sigma}
F(y)\frac{\exp\left(-\coth t\left(\frac{1}{2}|\xi|_E^2+\tr y\right)\right)}{\sinh^{\frac{N}{2}}t}
\exp\left(\frac{\sqrt{2}}{\sinh t}(\phi(y^{\frac{1}{2}})\xi|\sigma)_E\right)
\Delta(y)^{\frac{N}{2r}-\frac{n}{r}}d\sigma dy\\
=&\frac{1}{\Gamma_{\Omega}(\frac{N}{2r})}\int_{\Omega}
F(y)\frac{\exp\left(-\coth t\left(\frac{1}{2}|\xi|_E^2+\tr y\right)\right)}{\sinh^{\frac{N}{2}}t}
\mathcal{I}_{\frac{N}{2r}}\left(Q\left(\frac{1}{\sqrt{2}\sinh t}\phi(y^{\frac{1}{2}})\xi\right)\right)
\Delta(y)^{\frac{N}{2r}-\frac{n}{r}}dy\\
=&\frac{1}{\Gamma_{\Omega}(\frac{N}{2r})}\int_{\Omega}
F(y)\frac{\exp\left(-\coth t\left(\frac{1}{2}\tr Q(\xi)+\tr y\right)\right)}{\sinh^{\frac{N}{2}}t}
\mathcal{I}_{\frac{N}{2r}}\left(\frac{1}{2\sinh^2t}P(y^{\frac{1}{2}})Q(\xi)\right)
\Delta(y)^{\frac{N}{2r}-\frac{n}{r}}dy\\
=&\tau_{\frac{N}{2r}}(t)F\left(\frac{1}{2}Q(\xi)\right) 
\end{align*}
where we used \cite[Proposition XVI.2.1]{FK} at the 3rd equality and \cite[Lemma XVI.2.2.(ii)]{FK} at the 4th, 6th 
equalities. 
Therefore $\tau_{\frac{N}{2r}}(t)$ coincides with the action of the Hermite semigroup on radial functions on $E$. 
\end{remark}
\begin{remark}
For $x\in\overline{\mathcal{X}_1}$ (see (\ref{subvar})), 
$\mathcal{I}_{\lambda}(x)=\Gamma(\lambda)\tilde{I}_{\lambda-1}(2\sqrt{|x|_2})$ holds (see \cite[Example 3.3]{M}), 
and by analytic continuation the distribution 
$\frac{1}{\Gamma_{\Omega}(\lambda)}\Delta(x)^{\lambda-\frac{n}{r}}\mathbf{1}_{\Omega}dx$ 
at $\lambda=\frac{d}{2}$ gives the semi-invariant measure on $\overline{\mathcal{X}_1}\cap\overline{\Omega}$ 
(see \cite[Proposition VII.2.3]{FK}). Therefore for $V=\R^{1,n-1}$ the action $\tau_{\lambda}$ at 
$\lambda=\frac{d}{2}$ coincides with the action of the holomorphic semigroup on the minimal representation of 
$O(p,2)$ (see \cite[Theorem B]{KM1} or \cite[Theorem 5.1.1]{KM2}). 
\end{remark}
\begin{remark}
We set 
\[ H_{\lambda}\varphi(x):=i^{r\lambda}\tau_\lambda\left(\frac{\pi i}{2}\right)\varphi(x)
=\frac{1}{\Gamma_\Omega(\lambda)}\int_\Omega\varphi(y)\mathcal{J}\left(P(x^{\frac{1}{2}})y\right)
\Delta(y)^{\lambda-\frac{n}{r}}dy. \]
This is called the generalized Hankel transform (\cite[Section XV.4]{FK}). 
Similar to Remark \ref{Euclidrem}, this is regarded as a variant of the Fourier transform. 
Therefore it is expected that this Hankel transform has similar properties as the Fourier transform 
such as a Paley-Wiener type theorem, which determines the image of the compactly supported functions. 
This is done by, $e.g.$, \cite{A}, \cite[Remark 5.4]{K} for classical $V=\R$ case, but not for generalized case. 
In this paper we don't touch this topic in detail. 
\end{remark}

We set $K_{\lambda}(x,y;t):=e^{-\coth t(\tr x+\tr y)}\mathcal{I}_{\lambda}
\left(\sinh^{-2}tP(x^{\frac{1}{2}})y\right)$, the kernel function of $\tau_{\lambda}(t)$. Then 
we can deduce from Theorem \ref{upperestimate} that 

\begin{theorem}\label{kernelestimate}
Take $k\in\Z_{\ge 0}$ such that $\lambda+k>\frac{2n}{r}-1$. Then if $t=u+iv$, $u,v\in\R$, $u\ge 0$, 
\[ |K_{\lambda}(x,y;t)|\le C_{\lambda,t}\left(1+(\tr x\tr y)^\frac{rk}{2}\right)
\exp\left(-\frac{\sinh u}{\cosh u+|\cos v|}(\tr x+\tr y)\right). \]
\end{theorem}

Especially, if $u=\Re t>0$ then the integral defining $\tau_{\lambda}(t)$ converges 
if $\varphi$ is of polynomial growth, and the resulting $\tau_{\lambda}(t)\varphi$ has exponential decay. 
Even if $u=\Re t=0$, if $\lambda>\frac{2n}{r}-1$ and $t\notin \pi i\Z$, the integral converges if 
$\varphi\in L^1(\Omega,\Delta(x)^{\lambda-\frac{n}{r}}dx)$, and the resulting $\tau_{\lambda}(t)\varphi$ is bounded. 
In order to prove this theorem, we prepare the following lemma. 

\begin{lemma}
\begin{enumerate}
\item For $x\in\Omega$ the directional derivative of $x\mapsto\sqrt{x}$ is 
\[ D_{u}\sqrt{x}=\frac{1}{2}L\left(\sqrt{x}\right)^{-1}u. \]
\item For $x,y\in V$ if $[L(x),L(y)]=0$, then there exists a Jordan frame $\{c_1,\ldots,c_r\}$ such that 
$x,y\in\R\mathchar`-\operatorname{span}\{c_1,\ldots,c_r\}$. 
\item For $x,y\in\Omega$, $\displaystyle \tr\sqrt{P(x^{\frac{1}{2}})y}\le\sqrt{\tr x\tr y}\le\frac{\tr x+\tr y}{2}$. 
\end{enumerate}
\end{lemma}

\begin{proof}
(1) $u=D_ux=D_u\left(\sqrt{x}\right)^2=2\sqrt{x}D_u\sqrt{x}=2L\left(\sqrt{x}\right)D_u\sqrt{x}$ 
and then $D_u\sqrt{x}=\frac{1}{2}L\left(\sqrt{x}\right)^{-1}u$ follows. \\
(2) See \cite[Lemma X.2.2]{FK}. \\
(3) The second inequality is clear. For the first inequality, we take $k_0\in K$ such that 
$\tr\sqrt{P(x^{\frac{1}{2}})ky}$ ($k\in K_L$) attains its maximum at $k=k_0$. 
We put $k_0y=:y_0$. Then for any $D\in\mathfrak{k_l}=Lie(K_L)$, 
\begin{align*}
0=&\left.\frac{d}{dt}\right|_{t=0}\tr\sqrt{P(x^{\frac{1}{2}})e^{tD}y_0}
=\frac{1}{2}\tr\left(L\left(\sqrt{P(x^{\frac{1}{2}})y_0}\right)^{-1}P(x^{\frac{1}{2}})Dy_0\right) \\
=&\frac{1}{2}\left(\left.\sqrt{P(x^{\frac{1}{2}})y_0}^{-1}\right|P(x^{\frac{1}{2}})Dy_0\right)
=\frac{1}{2}\left(\left.P(x^{\frac{1}{2}})\sqrt{P(x^{\frac{1}{2}})y_0}^{-1}\right|Dy_0\right). 
\end{align*}
We put $P(x^{\frac{1}{2}})\sqrt{P(x^{\frac{1}{2}})y_0}^{-1}=:z$. If $D=[L(u),L(v)]$ ($u,v\in V$), then 
\begin{align*}
0=&(z|[L(u),L(v)]y_0)=(z|u(vy_0))-(z|v(uy_0))=(zu|vy_0)-(zv|uy_0) \\
=&(y_0(zu)|v)-(v|(uy_0)z)=([L(y_0),L(z)]u|v). 
\end{align*}
Since $(\cdot|\cdot)$ is non-degenerate, $[L(y_0),L(z)]=0$. Also, 
\begin{align*}
P(z)y_0&=P\left(P(x^{\frac{1}{2}})\sqrt{P(x^{\frac{1}{2}})y_0}^{-1}\right)y_0 \\
&=P(x^{\frac{1}{2}})P\left(\sqrt{P(x^{\frac{1}{2}})y_0}^{-1}\right)P(x^{\frac{1}{2}})y_0
=P(x^{\frac{1}{2}})e=x. 
\end{align*}
So especially $[L(x),L(y_0)]=0$. Let $x=\sum_{j=1}^rt_jc_j$, $y=\sum_{j=1}^rs_jd_j$ 
($t_j,s_j>0$, and $\{c_j\}_{j=1}^r,\{d_j\}_{j=1}^r$ are Jordan frames). Then, 
\begin{align*}
\tr\sqrt{P(x^{\frac{1}{2}})y}&\le\tr\sqrt{P(x^{\frac{1}{2}})y_0}
=\tr\sqrt{P\left(\sum_{j=1}^rt_j^{\frac{1}{2}}c_j\right)\sum_{j=1}^rs_jc_j}\\
&=\sum_{j=1}^r\sqrt{t_js_j}
\le\sqrt{\left(\sum_{j=1}^rt_j\right)\left(\sum_{j=1}^rs_j\right)}=\sqrt{\tr x\tr y}
\end{align*}
and the proof is completed. 
\end{proof}

Now we are ready to prove Theorem \ref{kernelestimate}. 

\begin{proof}[Proof of Theorem \ref{kernelestimate}]
By Corollary \ref{upperestimate}, 
\begin{align*}
&|K_{\lambda}(x,y;t)|\le C_{\lambda}'e^{-\Re\coth t(\tr x+\tr y)}
\left(1+\left|\frac{1}{\sinh t}\sqrt{P(x^{\frac{1}{2}})y}\right|_1^{rk}\right)
e^{2\left|\Re\frac{1}{\sinh t}\sqrt{P(x^{\frac{1}{2}})y}\right|_1} \\
=&\,C_{\lambda}'e^{-\Re\coth t(\tr x+\tr y)}
\left(1+\frac{1}{|\sinh t|^{rk}}\tr\left(\sqrt{P(x^{\frac{1}{2}})y}\right)^{rk}\right)
e^{2\left|\Re\frac{1}{\sinh t}\right|\tr\left(\sqrt{P(x^{\frac{1}{2}})y}\right)} \\
\le&\,C_{\lambda,t}\exp\left(-\frac{\cosh u\sinh u}{\cosh^2u-\cos^2v}(\tr x+\tr y)\right)
\left(1+\sqrt{\tr x\tr y}^{rk}\right)\\
&\qquad\times\exp\left(\frac{\sinh u|\cos v|}{\cosh^2u-\cos^2v}(\tr x+\tr y)\right) \\
=&\,C_{\lambda,t}\left(1+(\tr x\tr y)^\frac{rk}{2}\right)\exp\left(-\frac{\sinh u}{\cosh u+|\cos v|}(\tr x+\tr y)\right)
\end{align*}
and this completes the proof. 
\end{proof}

\section*{Acknowledgements}
The author would like to thank T. Kobayashi for many helpful advices on the topic of this paper. 
He also thanks his colleagues for many helpful discussions.

\end{document}